\documentclass{amsart}
\usepackage{amscd,amssymb,amsmath,graphicx,verbatim}
\usepackage[dvips]{hyperref}
\usepackage[TS1,OT1,T1]{fontenc}
\usepackage[all]{xy}
\newcommand{\Gal}{\operatorname{Gal}}
\newcommand{\Res}{\operatorname{Res}}
\newcommand{\Irr}{\operatorname{Irr}}
\newcommand{\Tr}{\operatorname{Tr}}
\newcommand{\Nr}{\operatorname{Nr}}
\newcommand{\GL}{\operatorname{GL}}
\newcommand{\id}{\operatorname{id}}
\newcommand{\Th}{\operatorname{th}}
\newcommand{\F}{\mathcal{F}}
\newcommand{\OO}{\mathcal{O}}
\newcommand{\QQ}{{\mathbb Q}}

\newcommand{\ZZ}{{\mathbb Z}}
\newcommand{\NN}{{\mathbb N}}
\newcommand{\HH}{{\mathbb H}}
\newcommand{\eee}{\hfill$\Box$}
\newtheorem{theorem}{Theorem}[section]
\newtheorem{lemma}[theorem]{Lemma}

\newtheorem{corollary}[theorem]{Corollary}

\newtheorem{example}[theorem]{Example}
\begin{document}
\title{Explicit Constructions of the non-Abelian
$\mathbf{p^3}$-Extensions Over $\mathbf{\QQ}$}
\author{Oz Ben-Shimol}
\maketitle
%
%      -------- Abstract --------
%
\begin{abstract}
Let $p$ be an odd prime.
Let $F/k$ be a cyclic extension of degree $p$ and of characteristic different from $p$.
The explicit constructions of the non-abelian $p^{3}$-extensions over $k$, are induced by certain elements in ${F(\mu_{p})}^{*}$.
In this paper we let $k=\QQ$ and present sufficient conditions for these elements to be suitable for the constructions.
Polynomials for the non-abelian groups of order $27$ over $\QQ$ are constructed.
\end{abstract}
\maketitle
%
%      -------- Section A --------
%
\section{Introduction}
Let $p$ be an odd prime.
Let $F/\QQ$ be a $C_{p}$-extension, $K=\QQ(\zeta_{p})$, $L=F(\zeta_{p})=FK$, where $\zeta_{p}$ is a fixed $p^{\Th}$ root of unity.
Let $\bar{\sigma}$ be a generator for $\Gal(L/K)$. The restriction $\sigma=\bar{\sigma}|_{F}$ generates $\Gal(F/\QQ)\cong\Gal(L/K)$.
Similarly, let $\bar{\tau}$ be a generator for $\Gal(L/F)$.
The restriction $\tau=\bar{\tau}|_{K}$ generates $\Gal(K/\QQ)\cong\Gal(L/F)$.
$L/\QQ$ is cyclic of order $p(p-1)$, generated by $\bar{\sigma}\bar{\tau}$.
\\
\indent
Let $e$ be a primitive root modulo $p$.
We have the multiplicative homomorphisms [4, p.130, Corollary 8.1.5]
$$
    \left\{
    \begin{array}{l}
        \Phi=\Phi_{L/F}:L^{*}\to L^{*} \\
        \Phi_{L/F}(x)=x^{e^{p-2}}\bar{\tau}x^{e^{p-3}}\cdots\bar{\tau}^{p-2}x, \ \ \ x\in L^{*},
    \end{array}
    \right.
$$
and its restriction $\Phi_{K/\QQ}:K^{*}\to K^{*}$. Clearly, $\Phi$ commutes with the norm:
$$
\Nr_{L/K}\Phi_{L/F}=\Phi_{K/\QQ}\Nr_{L/K}.
$$
\indent Let $\HH_{p^{3}}$ and $C_{p^{2}}\rtimes C_{p}$ be the non-abelian groups of order $p^3$.
From the generic construction of the non-abelian $p^{3}$-extensions (over fields with characteristic different from $p$), it follows that the elements $x\in L^{*}$ with $\Phi(\Nr_{L/K}(x))\not\in {L^{*}}^p$ induce $\HH_{p^{3}}$-extensions over $\QQ$, while the elements $x\in L^{*}$ with $\Phi(\zeta_{p}\Nr_{L/K}(x))\not\in{L^{*}}^p$ induce $C_{p^{2}}\rtimes C_{p}$-extensions over $\QQ$, cf. Theorem 4.1 and Theorem 4.2.
\\
\indent
Arne Ledet, in [3],[4], and [5, pages 130-132], constructs explicit extensions and the corresponding polynomials over $\QQ$ for these groups, when $p=3$ or $p=5$, and $L=\QQ(\zeta_{p^2})$.
He provides a sufficient condition for an element $x\in\OO_{L}$ to satisfy the above requirements, namely:
$\Nr_{L/\QQ}(x)=q$ is an unramified prime in $L$ (that is, $q\neq 3$, $q\neq 5$, respectively).
His argument for the sufficiency relies on unique factorization in $\OO_{K}$.
\\
\indent
Our goal is to generalize this idea for every $p$ and for every $C_{p}$-extension $F/\QQ$ the construction starts with (in [4], the construction starts with a $C_3$-extension inside $\QQ_{9}$).
More precisely, given a $C_{p}$-extension $F/\QQ$, we provide sufficient conditions for an element $x\in L^{*}$ to satisfy $\Phi(\Nr_{L/K}(x))\not\in{L^{*}}^p$ for $\HH_{p^{3}}$, and $\Phi(\zeta_{p}\Nr_{L/K}(x))\not\in{L^{*}}^p$ for $C_{p^{2}}\rtimes C_{p}$.
In section 4 we explain the explicit constructions for those two groups, and show that there are infinitely many elements $x\in L^{*}$, suitable for both constructions.
In section 5, we exhibit other polynomials for $\HH_{27}$ and $C_{9}\rtimes C_{3}$ over~$\QQ$.
%
%%%%%%%%%%%%%%%%%%%%%%%%%%%%%%%%%%%%
%     Notation and Definitions
%%%%%%%%%%%%%%%%%%%%%%%%%%%%%%%%%%%%
%
\section{Notation and Definitions}
\noindent
$\bullet$ \ The number ring of a given number field $N$ is denoted by $\OO_N$.
\vskip 0.1cm \noindent
$\bullet$ \ If $E$ is a finite group and $N/k$ is a Galois extension with Galois group $\Gal(N/k)\cong E$, we then say that $N/K$ is an $E$-extension, or that $N$ is an $E$-extension over $k$.
Any polynomial $f$ over $k$ with splitting field $N$ is called $E$-polynomial over $k$.
\vskip 0.1cm \noindent
$\bullet$ \ A cyclic group of order $m$ will denoted by $C_m$.
\vskip 0.1cm \noindent
$\bullet$ \ As in the introduction:
$F/\QQ$ is a $C_p$-extension, $K=\QQ(\zeta_{p})$ and $L=FK$ -- their compositum.
Let $\bar{\sigma}$ be a fixed generator for $\Gal(L/K)$. Thus $\sigma=\bar{\sigma}|_{F}$ generates $\Gal(F/\QQ)\cong\Gal(L/K)$.
Let $\bar{\tau}$ be a fixed generator for $\Gal(L/F)$. Thus $\tau=\bar{\tau}|_{K}$ generates $\Gal(K/\QQ)\cong\Gal(L/F)$.
\vskip 0.1cm \noindent
$\bullet$ \ If $e$ is a primitive root modulo $p$, then we have the homomorphism $\Phi$ defined in the introduction.
\vskip 0.1cm \noindent
$\bullet$ \ Let $\F_{L},\F_{K}$ be the multiplicative groups consisting of the fractional ideals of $L$,$K$, respectively.
We have analogous homomorphisms:
$$
    \left\{
    \begin{array}{l}
        \widetilde{\Phi}_{L/F}:\F_{L}\to \F_{L} \\
        \widetilde{\Phi}_{L/F}(I)=I^{e^{p-2}}\bar{\tau}I^{e^{p-3}}\cdots\bar{\tau}^{p-2}I, \ \ \ I\in\F_{L},
    \end{array}
    \right.
$$
and its restriction
$$
    \left\{
    \begin{array}{l}
        \widetilde{\Phi}=\widetilde{\Phi}_{K/\QQ}:\F_{K}\to \F_{K} \\
        \widetilde{\Phi}_{K/\QQ}(I)=\widetilde{\Phi}_{L/F}(\OO_{L}I)\cap K, \ \ \ I\in\F_{K}.
    \end{array}
    \right.
$$
Clearly, $\Nr_{L/K}\widetilde{\Phi}_{L/F}=\widetilde{\Phi}_{K/\QQ}\Nr_{L/K}$.
[remark: let $I\in\F_{L}$, then $\Nr_{L/K}(I)=\OO_{K}\cap
I\bar{\sigma}I\cdots\bar{\sigma}^{p-1}I$ (see Marcus [6] or Ribenboim [8])].
\vskip 0.1cm \noindent
$\bullet$ \ Let $x\in L^{*}$ with $\Nr_{L/\QQ}(x)=\pm q_{1}^{l_1}\cdots q_{n}^{l_n}$, \ $n>0$, \ $l_i\in\ZZ-\{0\}$, and the $q_i$'s are distinct rational primes.
We denote by $J_{i}(x)$, $i=1,\ldots n$, the fractional ideals of $L$ which satisfy:
$$
\OO_{L}x=J_{1}(x)\cdots J_{n}(x) \ \ \text{and} \ \ \Nr_{L/\QQ}(J_{i}(x))=\ZZ q_{i}^{l_{i}}.
$$
$\bullet$ \ For each $i$ we set $I_{i}(x):=\Nr_{L/K}(J_i(x))$. Clearly, $\Nr_{K/\QQ}(I_{i}(x))=\ZZ q_{i}^{l_{i}}$.
\vskip 0.1cm \noindent
$\bullet$ \ If for some $i$, $q_{i}$ splits completely in $K$ and $I_{i}(x)=P_{1}^{\beta_{1}}P_{2}^{\beta_{2}}\cdots P_{p-1}^{\beta_{p-1}}$, where the $P_{j}$'s are the distinct prime ideals in $K$ lying above $q_{i}$, then
we define $\chi(I_{i}(x))$ to be the integer
$$
    \chi(I_{i}(x))=e^{p-2}\beta_{1}+e^{p-3}\beta_{p-1}+\ldots e\beta_{3}+\beta_{2}.
$$
%
%%%%%%%%%%%%%%%%%%%%%%%%%%%%%%%%%%%%
%     The Main Results
%%%%%%%%%%%%%%%%%%%%%%%%%%%%%%%%%%%%
%
\section{The Main Results}
%
%
%%%%%%%%%%%%%%%%
%   Theorem
%%%%%%%%%%%%%%%%
%
\begin{theorem}\label{criteria}
Let $x\in L^{*}$, \ $\Nr_{L/\QQ}(x)=\pm q_{1}^{l_1}\cdots q_{n}^{l_n}$, \ $n\geq 0$, \ $l_i\in\ZZ-\{0\}$, \ $q_i$ a prime.
Then $\Phi(\Nr_{L/K}(x))$ does not generate a $p^{\Th}$ power (fractional) ideal of $L$,
if and only if $n>0$ and there exists an $i$ such that \\
\textbf{(a)} \ $q_{i}$ splits completely in $L$. \\
\textbf{(b)} \ $p$ does not divide $\chi(I_{i}(x))$.
\end{theorem}
%
%
%%%%%%%%%%%%%%%%%%%%%%%%%%%%%%%%%%%%
%   Proof of Theorem
%%%%%%%%%%%%%%%%%%%%%%%%%%%%%%%%%%%%
%
\begin{proof}
We assume $n>0$. Denote $\gamma=\Nr_{L/K}(x)$.
Clearly, $\OO_{K}\gamma=I_{1}(x)\cdots I_{n}(x)$. \
No prime factor of $I_{i}(x)$ is conjugate to any prime factor of $I_j(x)$ \ $(i\neq j)$.
It follows that $\OO_{L}\Phi(\gamma)\in\F_{L}^{p}$ \ if and only if $\OO_{L}\widetilde{\Phi}(I_{i}(x))\in\F_{L}^p$ for every $i=1,\ldots,n$.
\\
\indent
We fix some $i$ and set $q=q_i$, $l=l_i$, $J=J_i(x)$, $I=I_i(x)$. Thus, $J$ is a fractional ideal of $L$, $I=\Nr_{L/K}(J)$, and \ $\Nr_{L/\QQ}(J)=\Nr_{K/\QQ}(I)=\ZZ q^{l}$.
\vskip 0.1cm
We separate the rest of the proof into two cases.
\vskip 0.1cm\noindent
\textbf{Case 1: q does not split completely in L}
\\ \indent
Let $g$ be the decomposition number of $q$ in $K/\QQ$.
Clearly, $\tau^{g}I=I$.
Now,
$$
\widetilde{\Phi}(I)=\prod_{j=0}^{p-2}\tau^{j}I^{e^{p-(j+2)}}=
\prod_{j=1}^{\frac{p-1}{g}}\prod_{i=0}^{g-1}\tau^{i}I^{e^{p-((j-1)g+i+2)}}=
\prod_{i=0}^{g-1}\tau^{i}I^{S_{g}(i)},
$$
where,
$$
    S_g(i)=\sum_{j=1}^{\frac{p-1}{g}}e^{p-((j-1)g+i+2)}
    =e^{p-(i+2)}\cdot\frac{e^{-(p-1)}-1}{e^{-g}-1}.
$$
We see that if $q$ does not split completely in $K$ then $S_{g}(i)\equiv 0\mod p$ for all $i=0,\ldots g-1$.
Therefore, if $q$ does not split completely in $K$ then $\widetilde{\Phi}(I)\in\F_{K}^p$.
\\ \indent
Suppose that $q$ splits completely in $K$. 
If $q$ is ramified in $F$ then $q\equiv 1\mod p$ and any prime ideal $P$ in $K$ lying above $q$ is then totally ramified in $L$, equivalently, $\OO_{L}P\in\F_{L}^{p}$, hence $\OO_{L}I\in\F_{L}^p.$.
\\ \indent
Finally, suppose that $q$ is inert in $F$.
Then any prime in $K$ lying above $q$ is inert in $L$.
Hence, $\bar{\sigma}J=J$, thus $\OO_{L}I=J^{p}\in\F_{L}^{p}$.
\\ \indent
We conclude that if $q$ does not split completely in $L$, then $\OO_{L}\widetilde{\Phi}(I)\in\F_{L}^{p}$.
\vskip 0.1cm \noindent
\textbf{Case 2: q splits completely in L}
\\ \indent
Suppose that $q$ splits completely in $L$, and let $\OO_{K}q=P_{1}\ldots P_{p-1}$ be the prime decomposition of $q$ in $K$.
$\tau$, as a permutation acting on the $(p-1)$-set consisting of the $P_{j}$'s, is the $(p-1)$-cycle $(P_1,\ldots,P_{p-1})$ \ (say).
We shall denote by $\check{\tau}$ the corresponding $(p-1)$-cycle in the symmetric group; $\check{\tau}=(1,\ldots,p-1)$.
Now, if
$$
    I=P_{1}^{\beta_{1}}P_{2}^{\beta_{2}}\cdots P_{p-1}^{\beta_{p-1}},
$$
then
\begin{equation}
    \widetilde{\Phi}(I)=\prod_{j=1}^{p-1}\Phi(P_{j})^{\beta_j}=\prod_{j=1}^{p-1}\prod_{k=0}^{p-2}\bar{\tau}^{k}P_{j}^{e^{(p-(k+2))}\beta_{j}}
    =\prod_{j=1}^{p-1}P_{j}^{\chi_{j}(I)},
\end{equation}
where,
$$
    \chi_{j}(I)=\sum_{k=0}^{p-2}e^{(p-(k+2))}\beta_{\check{\tau}^{-k}(j)}, \ \ \ j=1,\ldots,p-1.
$$
Clearly, the primes $P_{j}$ split completely in L, hence, $\OO_{L}\widetilde{\Phi}(I)\in\F_{L}^{p}$ if and only if $\widetilde{\Phi}(I)\in\F_{K}^{p}$.
Also, $\widetilde{\Phi}(I)\in\F_{K}^{p}$ if and only if $\chi_{j}(I)\equiv 0\mod p$ \ for all $j=1,\ldots,p-1$.
Note that $\chi_{j+1}(I)\equiv e\cdot \chi_{j}(I)\mod p$ \ for all $j$. Therefore, if $p$ divides some $\chi_{j}(I)$, then $p$ must divides $\chi_{1}(I),\ldots,\chi_{p-1}(I)$ .
Since we consider the $\chi_{j}(I)$'s modulo $p$, we take $j=1$ and write $\chi(I)$ instead of $\chi_{1}(I)$.
We conclude that $\OO_{L}\widetilde{\Phi}(I)\in\F_{L}^{p}$ if and only if $p$ divides $\chi(I)$, as required.
\\ \indent
If $n=0$ then the assertion is clear.
\end{proof}
%
%%%%%%%%%%%%%%%%
%  Corollary
%%%%%%%%%%%%%%%%
%
\begin{corollary}
Let $x\in L^{*}$, \ $\Nr_{L/\QQ}(x)=\pm q_{1}^{l_1}\cdots q_{n}^{l_n}$, \ $n\geq 1$, \ $l_i\in\ZZ-\{0\}$, \ $q_i$ a prime.
Suppose that there exists an $i$ such that $q_{i}$ splits completely in $L$ and $p$ does not divide $\chi(I_{i}(x))$.
Then $\Phi(\Nr_{L/K}(x)), \Phi(\zeta_{p}\Nr_{L/K}(x))\not\in {L^{*}}^p$. \eee
\end{corollary}
Let $x\in L^{*}$ which satisfies condition (a) of Theorem \ref{criteria}. Let $q_{i}$ be a prime which splits completely in $L$ with $\Nr_{K/\QQ}(I_{i}(x))=\ZZ q_{i}^{l_{i}}$.
If $P$ is a prime ideal in $K$ lying above $q_{i}$ then
$\Nr_{K/\QQ}(P)=\ZZ q_{i}$ (since $q_{i}$ splits completely in $K$). Therefore, if $I_{i}(x)=P_{1}^{\beta_{1}}P_{2}^{\beta_{2}}\cdots P_{p-1}^{\beta_{p-1}}$,
where the $P_{j}$'s are the distinct prime ideals in $K$ lying above $q_{i}$, then
$$
    \ZZ q_{i}^{l_{i}}=\Nr_{K/\QQ}(I_{i}(x))=
    \prod_{j=1}^{p-1}\Nr_{K/\QQ}P_{j}^{\beta_{j}}=\ZZ q_{i}^{\beta_{1}+\ldots\beta_{p-1}}.
$$
Thus,
\begin{equation}
\beta_{1}+\ldots+\beta_{p-1}=l_{i}.
\end{equation}
It follows immediately that if $I_{i}(x)$ is of the form $I_{i}(x)=P^{l}$ (for some prime ideal $P$ in $K$ lying above $q_{i}$ and for some $l\in\ZZ-\{0\})$, then $p$ divides $\chi(I_{i}(x))$ if and only if $p$ divides $l$.
We have
%
%
%%%%%%%%%%%%%%%%
%  Corollary
%%%%%%%%%%%%%%%%
%
\begin{corollary}\label{suffP}
Let $x\in L^{*}$, \ $\Nr_{L/\QQ}(x)=\pm q_{1}^{l_1}\cdots q_{n}^{l_n}$, \ $n\geq 1$, \ $l_i\in\ZZ-\{0\}$, \ $q_i$ a prime.
Suppose that there exists an $i$ such that $q_{i}$ splits completely in $L$ and $I_{i}(x)=P^{l_{i}}$,
where $P$ is a prime ideal of $\OO_{K}$, $p$ does not divide $l_{i}$.
Then $\Phi(\Nr_{L/K}(x)), \Phi(\zeta_{p}\Nr_{L/K}(x))\not\in {L^{*}}^p$. \eee
\end{corollary}
Let $\mathcal{H}(L)$ be the Hilbert class field over $L$ (see [6, Chapter 8]), and let $q$ be a rational prime which splits completely in $\mathcal{H}(L)$.
Then any prime ideal of $L$ lying above $q$ is principal in $L$.
%
%%%%%%%%%%%%%%%%
%  Corollary
%%%%%%%%%%%%%%%%
%
\begin{corollary}\label{inf}
There are infinitely many elements $x\in L^{*}$ such that \\
1. \ $\Phi(\Nr_{L/K}(x))\not\in {L^{*}}^p$, and \\
2. \ $\Phi(\zeta_{p}\Nr_{L/K}(x))\not\in {L^{*}}^p$.
\end{corollary}
%
%%%%%%%%%%%%%%%%%%%%%%%
%  Proof of Corollary
%%%%%%%%%%%%%%%%%%%%%%%
%
\begin{proof}
By Chebotarev's Density Theorem, there are infinitely many rational primes which split completely in $\mathcal{H}(L)$.
If $q$ is such a prime, let $Q$ be a prime ideal in $L$ lying above $q$. So $Q=\OO_L x$ for some $x\in \OO_{L}$. Clearly, $\Nr_{L/\QQ}(x)=q$.
The assertion then follows from Corollary \ref{suffP}.
\end{proof}
%
%%%%%%%%%%%%%%%%
%  Examples
%%%%%%%%%%%%%%%%
%
\subsection{Examples}
We shall illustrate the above results on $C_{p}$-extensions $F/\QQ$ of type $(S_1)$ [10, Chap. 2]. For that purpose we prove the following lemma.
%
%%%%%%%%%%%%%%%%
%  Lemma
%%%%%%%%%%%%%%%%
%
\begin{lemma}
Let $r$ be a rational prime, $r\equiv 1\mod k$, $k\in\NN$.
Let $m=m(r)$ be a primitive root modulo $r$.
Consider the sum
\begin{equation}
    \delta_{k}(r)=
    \zeta_{r}+\zeta_{r}^{m^{k}}+\zeta_{r}^{m^{2k}}+\ldots+
    \zeta_{r}^{m^{\left(\frac{r-1}{k}-1\right)k}}.
\end{equation}
Then $\QQ(\delta_{k}(r))/\QQ$ is a $C_{k}$-extension.
\end{lemma}
%
%%%%%%%%%%%%%%%%%%%
%  Proof of Lemma
%%%%%%%%%%%%%%%%%%%
%
\begin{proof}
$\Gal(\QQ(\zeta_{r})/\QQ)$ is cyclic of order $r-1$, generated by the automorphism $\sigma:\zeta_{r}\mapsto\zeta_{r}^{m}$.
By the Fundamental Theorem of Galois Theory, it is enough to prove $\QQ(\delta_{k}(r))=\QQ(\zeta_{r})^{\sigma^{k}}$.
The inclusion $\QQ(\delta_{k}(r))\subseteq\QQ(\zeta_{r})^{\sigma^{k}}$ is because $\sigma^{k}$ moves cyclically the summands of (3) (an alternative argument: $\delta_{k}(r)$ is the image of $\zeta_{r}$ under the trace map $\Tr_{\QQ(\zeta_{r})/\QQ(\zeta_{r})^{\sigma^k}}$).
Suppose that $\QQ(\delta_{k}(r))\subsetneqq\QQ(\zeta_{r})^{\sigma^{k}}$.
There exist a proper divisor $d$ of $k$ such that $\QQ(\delta_{k}(r))=\QQ(\zeta_{r})^{\sigma^{d}}$.
In particular, $\sigma^{d}(\delta_{k}(r))=\delta_{k}(r)$, or
\begin{equation}
    \sum_{j=0}^{\frac{r-1}{k}-1}\zeta_{r}^{m^{jk+d}}-
    \sum_{j=0}^{\frac{r-1}{k}-1}\zeta_{r}^{m^{jk}}=0.
\end{equation}
The summands in (4) are distinct.
For if $\zeta_{r}^{m^{jk+d}}=\zeta_{r}^{m^{ik}}$ for some $i,j=0,1,\ldots,\frac{r-1}{k}-1$, $j\geq i$, then $m^{(j-i)k+d}\equiv 1(\mod r)$.
$m$ is primitive modulo $r$ so $r-1$ divides $(j-i)k+d$.
But, $(j-i)k+d<(\frac{r-1}{k}-1)k+k=r-1$, a contradiction.
To this end, there are $2(r-1)/k$ ($\leq r-1$) summands, and dividing each of them by $\zeta_{r}$ gives us a linear dependence among $1,\zeta_{r},\zeta_{r}^2,\ldots,\zeta_{r}^{r-2}$, a contradiction. $\QQ(\delta_{k}(r))=\QQ(\zeta_{r})^{\sigma^{k}}$, as required.
\end{proof}
%
%%%%%%%%%%%%%%%%
%  Example 1
%%%%%%%%%%%%%%%%
%
\begin{example}
\emph{
$p=3$, $r=7$, $m(r)=3$. \
$\delta_{3}(7)=\zeta_{7}+\zeta_{7}^{-1}$. \
$F/\QQ=\QQ(\delta_{3}(7))/\QQ$.
Thus, $L/K=\QQ(\zeta_{3},\delta_{3}(7))/\QQ(\zeta_{3})$ is a
$C_{3}$-extension, generated by $\bar\sigma:\zeta_{7}\mapsto\zeta_{7}^{2}$, $\zeta_{3}\mapsto\zeta_{3}$.
Let $x=\delta_{3}(7)+\zeta_{3}$. \
$\gamma=\Nr_{L/K}(x)=3-\zeta_{3}$, \
$\Nr_{L/\QQ}(x)=\Nr_{K/\QQ}(\gamma)=13$ \ ($\tau:\zeta_{3}\mapsto\zeta_{3}^{-1}$).
$13$ splits
completely in $L$. By Corollary \ref{suffP}, $\Phi(\gamma)$, $\zeta_{3}\Phi(\gamma)\not\in {L^{*}}^{3}$.
}
\end{example}
%
%%%%%%%%%%%%%%%%
%  Example 2
%%%%%%%%%%%%%%%%
%
\begin{example}
\emph{
$p=3$, $r=19$, $m(r)=2$.
$$
\delta_{3}(19)=\zeta_{19}+\zeta_{19}^{-1}+\zeta_{19}^{7}+\zeta_{19}^{-7}+\zeta_{19}^{8}+\zeta_{19}^{-8}.
$$
$F/\QQ=\QQ(\delta_{3}(19))/\QQ$. Thus,
$L/K=\QQ(\zeta_{3},\delta_{3}(19))/\QQ(\zeta_{3})$ is a
$C_{3}$-extension, generated by
$\bar\sigma:\zeta_{19}\mapsto\zeta_{19}^{6}$,
$\zeta_{3}\mapsto\zeta_{3}$. Let $x=\delta_{3}(19)+\zeta_{3}+1$. \
$\gamma=\Nr_{L/K}(x)=-7\zeta_3$, \
$\Nr_{L/\QQ}(x)=\Nr_{K/\QQ}(\gamma)=7^{2}$ and \ $7$ splits completely
in $L$. However, $\Phi(\gamma)$ generates a third power ideal in
$\OO_K$ although $\Phi(\gamma)$ and $\zeta_{3}\Phi(\gamma)$ are not third power elements in $L^{*}$. Therefore, Theorem \ref{criteria}, which is an \emph{ideal-theoretic} criterion, is failed while testing this $x$.
On the other hand, it tells us something
about the prime factorization of $\OO_{L}x$ in $L$: Since $x$ satisfies condition (a) of this theorem and does not satisfy condition (b), we have
$$
\chi(I_{1}(x))=2\beta_{1}+\beta_{2}\equiv 0\mod 3.
$$
Also, by equation (2) we have $\beta_{1}+\beta_{2}=2$.
Hence, $\beta_{1}=\beta_{2}=1$, thus $I=\OO_{K}7$.
$\OO_{L}x$ is therefore a product of two distinct primes of $O_{L}$, each of them has norm $7$ over $\QQ$.
}
\end{example}
%
%%%%%%%%%%%%%%%%
%  Example 3
%%%%%%%%%%%%%%%%
%
\begin{example}
\emph{
$p=3$, $r=73$, $m(r)=5$. (Remark: $\OO_{K}=\ZZ[\zeta_{73}]$ is not a unique factorization domain).
$F/\QQ=\QQ(\delta_{3}(73))/\QQ$.
Thus, $L/K=\QQ(\zeta_{3},\delta_{3}(73))/\QQ(\zeta_{3})$ is a
$C_{3}$-extension, generated by $\bar\sigma:\zeta_{73}\mapsto\zeta_{73}^{24}$,  $\zeta_{3}\mapsto\zeta_{3}$.
Let $x=\delta_{3}(73)-\zeta_{3}+1$. \
$\gamma=\Nr_{L/K}(x)=21\zeta_{3}$, \
$\Nr_{L/\QQ}(x)=\Nr_{K/\QQ}(\gamma)=3^{2}7^{2}$. \
Thus, $\Phi(\gamma)=21^{3}\zeta_{3}\not\in {L^{*}}^{3}$.
}
\end{example}
%
%%%%%%%%%%%%%%%%
%  Example 4
%%%%%%%%%%%%%%%%
%
\begin{example}
\emph{
$p=5$, $r=11$, $m(r)=2$. \
$\delta_{5}(11)=\zeta_{11}+\zeta_{11}^{-1}$. \
$F/\QQ=\QQ(\delta_{5}(11))/\QQ$.
Thus, $L/K=\QQ(\zeta_{5},\delta_{5}(11))/\QQ(\zeta_{5})$ is a
$C_{5}$-extension, generated by $\bar\sigma:\zeta_{11}\mapsto\zeta_{11}^{2}$, $\zeta_{5}\mapsto\zeta_{5}$.
Let $x=\delta_{5}(11)-\zeta_{5}$. \
$\gamma=\Nr_{L/K}(x)$,
$\Nr_{L/\QQ}(x)=\Nr_{K/\QQ}(\gamma)=991$ ($991$ is a rational prime).
$991$ splits
completely in $L$, hence $\Phi(\gamma)$ does not generate a fifth power ideal of $\OO_L$.
}
\end{example}
%
%%%%%%%%%%%%%%%%%%%%%%%%%%%%%%%%%%%%%%%%%%%%%%%%%%%%%%%%%%%%%
%  Explicit Constructions of the non-Abelian $p^3$-Extensions
%%%%%%%%%%%%%%%%%%%%%%%%%%%%%%%%%%%%%%%%%%%%%%%%%%%%%%%%%%%%%
%
\section{Explicit Constructions of the non-Abelian $p^3$-Extensions}
\subsection{}
Let $N/k$ be a Galois extension with Galois group $G=\Gal(N/k)$.
Let $\pi:E\twoheadrightarrow G$ be an epimorphism of a given finite group $E$ onto $G$.
\emph{The Galois theoretical embedding problem} is to find a pair $(T/k,\varphi)$ consisting of a Galois extension $T/k$ which contains $N/k$, and an isomorphism $\varphi:\Gal(T/k)\to E$ such that $\pi\circ\varphi=\Res_{T/N}$, where $\Res_{T/N}:\Gal(T/k)\twoheadrightarrow G$ is the restriction map.
We say that the embedding problem is \emph{given by $N/k$ and $\pi:E\twoheadrightarrow G$}.
If such a pair $(T/k,\varphi)$ does exist, we say that the embedding problem is \emph{solvable}.
$T$ is called a \emph{solution field}, and $\varphi$ is called a \emph{solution isomorphism}. $A:=\ker\pi$ is called the \emph{kernel of the embedding problem}.
If $A$ is cyclic of order $m$, the characteristic of $k$ does not divide $m$, and $N^{*}$ contains $\mu_{m}$ (the multiplicative group consisting of the $m^{\Th}$ roots of unity), then the embedding problem is of \emph{Brauer type} (see [4]).
\\
\indent
If $\varphi$ is required only to be a monomorphism, we say that the embedding problem is \emph{weakly solvable}.
In that case, the pair $(T/k,\varphi)$ is called an \emph{improper solution} (or a \emph{weak solution}).
An improper solution can also be interpret in terms of Galois algebra ([1], [3] Chapter 4).
The monograph [2] treats the embedding problem from this point of view.
In the cases discussed in this paper, any improper solution will automatically be a solution;
the kernel $A$ will be contained in the Frattini subgroup of $E$ (see [2], Chapter.6, Corollary 5).
\\
\indent
The constructive approach to embedding problems is to describe explicitly the solution field $T$ (in case it exists).
We shall consider a series of embedding problems, in each $T/N$ will be a $C_{p}$-extension, and the goal is to exhibit an explicit primitive element of $T$ over $N$ (See also [7] for further details).
In practice, given a set of generators $\gamma_{1},\ldots,\gamma_{d}$ for $G$, we must extend them to $k$-automorphisms $\bar{\gamma_{1}},\ldots,\bar{\gamma_{d}}$ of a larger field $T$, with common fixed field $k$.
$T/k$ is then a Galois extension generated by $\bar{\gamma_{1}},\ldots,\bar{\gamma_{d}},\bar{\gamma}_{d+1},\ldots,\bar{\gamma_{c}}$, where $\bar{\gamma}_{d+1},\ldots,\bar{\gamma_{c}}$ generate $\Gal(T/N)$.
The extensions should be constructed in such a way that the $\bar{\gamma_{i}}$'s behave as a given set of $c$ generators for $E$.
$T$ is then a solution field, and the isomorphism which takes each $\bar{\gamma_{i}}$ to the corresponding generator of $E$ is a solution isomorphism.
\subsection{}
Let $N$ be a field of characteristic different from $p$.
$\mu_{p}\subseteq N^{*}$, and let $a\in N^{*}\setminus{N^{*}}^{p}$.
Then $N(\sqrt[p]{a})/N$ is a $C_{p}$-extension.
A generator for $\Gal(N(\sqrt[p]{a})/N)$ is given by $\sqrt[p]{a}\mapsto\zeta_{p}\sqrt[p]{a}$.
The only elements $\theta\in N(\sqrt[p]{a})$ with the property: $\theta^{p}\in N$, are those of the form $z(\sqrt[p]{a})^{j}$ for some $z\in N$, $j=0,1,\ldots,p-1$.
\\
\indent
Suppose further that $N/k$ is a Galois extension, then $N(\sqrt[p]{a})/k$ is a Galois extension if and only if, for each $\gamma\in\Gal(N/k)$, there exist $i_{\gamma}\in\ZZ\setminus p\ZZ$ \ such that \ $\gamma{a}/a^{i_{\gamma}}\in {N^{*}}^{p}$.
\subsection{}
Let $E$ be a non-abelian group of order $p^{3}$. Up to isomorphism, $E$ is necessarily one of the following two groups:
\newline
\textbf{1.} \ The Heisenberg group:
$$
    \HH_{p^{3}}=\langle \ u,v,w \ : \ u^{p}=v^{p}=w^{p}=1, \ wu=uw, \ wv=vw, \ vu=uvw  \ \rangle.
$$
In other words, $\HH_{p^{3}}$ is generated by two elements $u$, $v$ of order $p$, such that their commutator $w$ is central.
It can be realized as the subgroup of $\GL_{3}(\ZZ/\ZZ p)$ consisting of upper triangular matrices with 1's in the diagonal.
\\
\textbf{2.} \ The semidirect product:
$$
    C_{p^{2}}\rtimes C_{p}=\langle \ u,v \ : \ u^{p^{2}}=v^{p}=1, \ vu=u^{p+1}v \ \rangle.
$$
For the classification of the non-abelian groups of order $p^3$, see [9, page 67].
\subsection{}
The aim of the following Theorem 4.1 (resp. Theorem 4.2) is to describe \emph{how} $\HH_{p^3}$-extensions (resp. $C_{p^{2}}\rtimes C_{p}$-extensions) are constructed over $\QQ$, starting with elements $x\in L^{*}$ with
$\Phi(\Nr_{L/K}(x))\not\in {L^{*}}^{p}$ (resp. $\Phi(\zeta_{p}\Nr_{L/K}(x))\not\in {L^{*}}^{p}$).
These constructions are essentially known - they can be obtained as private cases of [4, Theorem 2.4.1 and Corollary 8.1.5].
We give a direct and self contained proof, consistent with our notation and with subsection 4.1.
\begin{theorem}[The case $E=\HH_{p^{3}}$] \
\\
Let $x\in L^{*}$ such that $b=b(x)=\Phi(\Nr_{L/K}(x))\not\in {L^{*}}^{p}$.
Then $G=\Gal(L(\sqrt[p]{b})/K)\cong C_p\times C_p$, generated by $\bar{\bar\sigma}$ and $\eta$, where, $\bar{\bar\sigma}: \sqrt[p]{b}\mapsto\sqrt[p]{b}$, \
$\eta:\sqrt[p]{b}\mapsto\zeta_{p}\sqrt[p]{b}$, \ $\bar{\bar\sigma}|_L=\bar\sigma$, \ \ $\eta|_{L}=\id_{L}$.
Moreover, the (Brauer type) embedding problem given by $L(\sqrt[p]{b})/K$ and $\pi:E\twoheadrightarrow G$ is solvable, where $\pi: u\mapsto\bar{\bar\sigma}$, $v\mapsto\eta$; \
a solution field is $M=L(\sqrt[p]{\omega},\sqrt[p]{b})$, \
$\omega=\Phi(\beta)$, \ $\beta=x^{p-1}\bar\sigma x^{p-2}\ldots\bar\sigma^{p-2}x$.
\\ \indent
Since $\omega\in\Phi(L^{*})$, $M/\QQ$ is a Galois extension with
$\Gal(M^{\kappa}/\QQ)\cong E$, where $\kappa$ is the extension of $\bar{\tau}$ to $M$, given by
$\kappa:\sqrt[p]{\omega}\mapsto\beta^{(1-e^{p-1})/p}\left({\sqrt[p]{\omega}}\right)^{e}$, and
$M^{\kappa}=F(\alpha)$, \ $\alpha=\sqrt[p]{\omega}+\kappa\sqrt[p]{\omega}+\ldots+\kappa^{p-2}\sqrt[p]{\omega}$.
Finally, $F(\alpha)$ is the splitting field of the minimal polynomial for
$\alpha$ over $\QQ$.
\end{theorem}
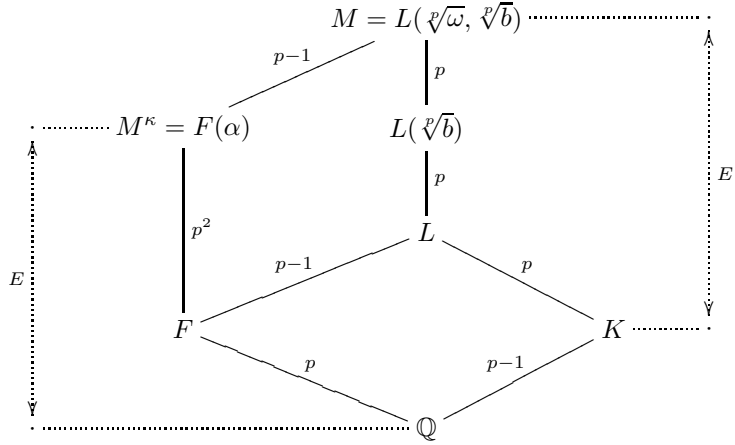
\begin{equation*}
    \xymatrix{
          &                       & M=L(\sqrt[p]{\omega},\sqrt[p]{b}) \ar@{-}[d]^{p} \ar@{-}[dl]_{p-1} \ar@{.}[rr] & & \cdot \ar@{<.>}[ddd]^{E} \\
          \cdot \ar@{<.>}[ddd]_{E} \ar@{.}[r] & M^{\kappa}=F(\alpha) \ar@{-}[dd]^{p^2} & L(\sqrt[p]{b}) \ar@{-}[d]^{p} & & \\
                               &  & L \ar@{-}[dl]_{p-1} \ar@{-}[dr]^{p} & & \\
          & F \ar@{-}[dr]^{p}          &  & K \ar@{-}[dl]_{p-1} \ar@{.}[r]      & \cdot \\
          \cdot \ar@{.}[rr] & & \QQ & & &
   }
\end{equation*}
\begin{proof}
Since $\bar{\sigma}b=b$, \ $L(\sqrt[p]{b})/K$ is a Galois extension.
We extend $\bar{\sigma}$ to
$\Gal(L(\sqrt[p]{b})/K)$ by $\bar{\bar{\sigma}}(\sqrt[p]{b})=\sqrt[p]{b}$.
$\eta$ generates $\Gal(L(\sqrt[p]{b})/L)$, and $\bar{\sigma}$ generates $\Gal(L/K)$.
It follows that $\Gal(L(\sqrt[p]{b})/K)$ is generated by $\bar{\bar{\sigma}}$ and $\eta$.
Clearly,
$\bar{\bar{\sigma}}\eta=\eta\bar{\bar{\sigma}}$ \ and \ $\bar{\bar{\sigma}}^{p}=\eta^{p}=\id_{L(\sqrt[p]{b})}$, thus,
$L(\sqrt[p]{b})/K$ is a $C_{p}\times C_{p}$-extension [$L(\sqrt[p]{b})$ is a Kummer extension of $K$ of exponent $p$].
\\
\indent
$\eta(\omega)=\omega$ \ and,
\begin{equation}
\frac{\bar{\bar{\sigma}}(\omega)}{\omega}=\frac{\bar{\sigma}(\Phi(\beta))}{\Phi(\beta)}=
\frac{\Phi(\Nr_{L/K}(x))}{\Phi(x)^{p}}=
\left(\frac{\sqrt[p]{b}}{\Phi(x)}\right)^{p}\in{L(\sqrt[p]{b})^{*}}^{p},
\end{equation}
thus, $M/K$ is a Galois extension.
We extend $\bar{\bar{\sigma}}$ and $\eta$ to
$\tilde{\sigma},\tilde{\eta}\in\Gal(M/K)$ by
\begin{equation}
\tilde{\sigma}(\sqrt[p]{\omega})=\frac{\sqrt[p]{b}}{\Phi(x)}\sqrt[p]{\omega}, \ \ \ \ \ \tilde{\eta}(\sqrt[p]{\omega})=\zeta_{p}\sqrt[p]{\omega},
\end{equation}
respectively.
Note that $M/L(\sqrt[p]{b})$ is a $C_{p}$-extension.
For if $\omega\in {L(\sqrt[p]{b})^{*}}^{p}$ then $\omega=z^{p}b^{j}$ for some $z\in L$, $j=0,1,\ldots p-1$.
From $(5)$ it follows that $b=(\Phi(x)(\bar{\sigma}(z)/z))^{p}\in{L^{*}}^{p}$ - a contradiction.
\\
\indent
$\Gal(M/K)$ is then generated by $\tilde{\sigma},\tilde{\eta}$ and $\tilde{\lambda}$,
where $\tilde{\lambda}(\sqrt[p]{\omega})=\zeta_{p}\sqrt[p]{\omega}$ - a generator for
$\Gal(M/L(\sqrt[p]{b}))$.
Clearly, $\tilde{\eta}$ and $\tilde{\lambda}$ is of order $p$, and since
\small
$$
    \tilde{\sigma}^{p}(\sqrt[p]{\omega})=
    \frac{b}{\Phi(x)\bar{\sigma}(\Phi(x))\cdots\bar{\sigma}^{p-1}(\Phi(x))}\sqrt[p]{\omega}=
    \frac{b}{\Nr_{L/K}(\Phi(x))}\sqrt[p]{\omega}=\frac{b}{\Phi(\Nr_{L/K}(x))}\sqrt[p]{\omega}=\sqrt[p]{\omega},
$$
\normalsize
$\tilde{\sigma}$ is of order $p$.
$\tilde{\lambda}$ commutes with both $\tilde{\sigma}$ and $\tilde{\eta}$, and the
relation $\tilde{\eta}\tilde{\sigma}=\tilde{\sigma}\tilde{\eta}\tilde{\lambda}$ holds.
Therefore, the group isomorphism $\varphi:\Gal(M/K)\cong\HH_{p^3}$,
sending $\tilde{\sigma}\mapsto u$, $\tilde{\eta}\mapsto v$, $\tilde{\lambda}\mapsto w$, is a solution isomorphism
for the given embedding problem.
\\
\indent
Now,
$$
\frac{\bar{\tau}\Phi(x)}{\Phi(x)^{e}}=\Phi\left(\frac{\bar{\tau}x}{x^e}\right)=
\frac{\bar{\tau}x^{e^{p-2}}}{x^{e^{p-1}}}\cdot
\frac{\bar{\tau}^{2}x^{e^{p-3}}}{\bar{\tau}x^{e^{p-2}}}\cdots
\frac{\bar{\tau}^{p-2}x^{e}}{\bar{\tau}^{p-3}x^{e^2}}\cdot
\frac{x}{\bar{\tau}^{p-2}x^{e}}=x^{1-e^{p-1}},
$$
hence,
$$
\frac{\bar{\tau}\omega}{\omega^e}=
\frac{\bar{\tau}\Phi\left(x^{p-1}\bar{\sigma}x^{p-2}\cdots\bar{\sigma}^{p-2}x\right)}
{\Phi\left(x^{p-1}\bar{\sigma}x^{p-2}\cdots\bar{\sigma}^{p-2}x\right)^e}=
\left(\frac{\bar{\tau}\Phi(x)}{\Phi(x)^e}\right)^{p-1}
\bar{\sigma}\left(\frac{\bar{\tau}\Phi(x)}{\Phi(x)^e}\right)^{p-2}\cdots \
\bar{\sigma}^{p-2}\left(\frac{\bar{\tau}\Phi(x)}{\Phi(x)^e}\right)
$$
$$
=(x^{p-1}\bar{\sigma}x^{p-2}\cdots\bar{\sigma}^{p-2}x)^{1-e^{p-1}}=
\left(\beta^{\frac{1-e^{p-1}}{p}}\right)^{p}\in {L^{*}}^{p},
$$
and
$$
    \frac{\bar{\tau}b}{b^{e}}=\frac{\bar{\tau}\Phi(\Nr_{L/K}(x))}{\Phi(\Nr_{L/K}(x))^{e}}=
    \Nr_{L/K}\left(\frac{\bar{\tau}\Phi(x)}{\Phi(x)^{e}}\right)=
    \left(\Nr_{L/K}(x)^{\frac{1-e^{p-1}}{p}}\right)^{p}\in {L^{*}}^{p}.
$$
Thus, $M/\QQ$ is a Galois extension.
We extend $\bar{\tau}$ to $\kappa\in\Gal(M/\QQ)$ by
$$
    \kappa(\sqrt[p]{\omega})=\beta^{(1-e^{p-1})/p}\left({\sqrt[p]{\omega}}\right)^{e}, \ \ \ \ \
    \kappa(\sqrt[p]{b})=\Nr_{L/K}(x)^{(1-e^{p-1})/p}(\sqrt[p]{b})^{e}.
$$
Since $\kappa$ is an extension of $\tau$ to $M$, $\Gal(M/\QQ)$ is generated by $\tilde{\sigma}$,$\tilde{\eta}$,$\tilde{\lambda}$, and $\kappa$.
It is easy to verify that $\kappa$ is central and has order $p-1$
[while doing it, we take $\tau:\zeta_{p}\mapsto\zeta_{p}^{e}$.
Also, note that $\kappa$ is the unique extension of $\bar{\tau}$ of order $p-1$].
It follows that $\Gal(M/\QQ)\cong\HH_{p^{3}}\times C_{p-1}$ and $M^{\kappa}/\QQ$ is Galois with
$\Gal(M^{\kappa}/\QQ)\cong\HH_{p^3}$, as required.
\\
\indent
Clearly, $F(\alpha)\subseteq M^{\kappa}$. \
We have the short exact sequence
$$
\xymatrix{
    1 \ar@{->}[r] & \Gal(M/F) \ar@{->}[r]^{\imath} &
    \Gal(M/\QQ) \ar@{->}[r]^{\Res} & \Gal(F/\QQ) \ar@{->}[r] & 1,
}
$$
It follows that $\Gal(M/F)$ is generated by $\tilde{\eta}$,$\tilde{\lambda}$, and $\kappa$.
Thus, $F(\alpha)=M^{\kappa}$ if and only if $\tilde{\eta}\alpha\neq\alpha$ and $\tilde{\lambda}\alpha\neq\alpha$.
But $\tilde{\eta}\alpha=\alpha$ or $\tilde{\lambda}\alpha=\alpha$ imply linear dependence among the elements
$\sqrt[p]{\omega},\sqrt[p]{\omega}^{e},\ldots,\sqrt[p]{\omega}^{e^{p-2}}$ over $L$, a contradiction.
\\
\indent
To this end, $\Gal(F(\alpha)/\QQ(\alpha))\cong\Gal(F/\QQ)\cong C_{p}$,
and $\QQ(\alpha)/\QQ$ is not a normal extension.
Hence, $F(\alpha)$ is the splitting field for $\Irr(\alpha;\QQ)$.
\end{proof}
\begin{theorem}[The case $E=C_{p^{2}}\rtimes C_{p}$] \
\\
Let $x\in L^{*}$ such that $b=b(x)=\Phi(\zeta_{p}\Nr_{L/K}(x))\not\in {L^{*}}^{p}$.
Then $G=\Gal(L(\sqrt[p]{b})/K)$ \\ $\cong C_p\times C_p$,
generated by $\bar{\bar\sigma}$ and $\eta$, where,
$\bar{\bar\sigma}: \sqrt[p]{b}\mapsto\sqrt[p]{b}$, \
$\eta:\sqrt[p]{b}\mapsto\zeta_{p}\sqrt[p]{b}$, \ $\bar{\bar\sigma}|_L=\bar\sigma$, \ \ $\eta|_{L}=\id_{L}$.
Moreover, the (Brauer type) embedding problem given by $L(\sqrt[p]{b})/K$ and
$\pi:E\twoheadrightarrow G$ is solvable, where $\pi: u\mapsto\bar{\bar\sigma}$, $v\mapsto\eta$; \
a solution field is $M=L(\sqrt[p]{\omega},\sqrt[p]{b})$, \
$\omega=\Phi(\beta\sqrt[p]{a})$, \ $\beta=x^{p-1}\bar\sigma x^{p-2}\ldots\bar\sigma^{p-2}x$, \ $L=K(\sqrt[p]{a})$.
\\ \indent
Since $\omega\in\Phi(L^{*})$, $M/\QQ$ is a Galois extension with
$\Gal(M^{\kappa}/\QQ)\cong E$, where $\kappa$ is the extension of $\bar{\tau}$ to $M$, given by
$\kappa:\sqrt[p]{\omega}\mapsto{(\beta\sqrt[p]{a})}^{(1-e^{p-1})/p}\left({\sqrt[p]{\omega}}\right)^{e}$, and
$M^{\kappa}=F(\alpha)$, \ $\alpha=\sqrt[p]{\omega}+\kappa\sqrt[p]{\omega}+\ldots+\kappa^{p-2}\sqrt[p]{\omega}$.
Finally, $F(\alpha)$ is the splitting field of the minimal polynomial for
$\alpha$ over $\QQ$.
\end{theorem}
\begin{proof}
A few changes should be made in the proof of Theorem 4.1. \\
We can assume $\bar{\sigma}:\sqrt[p]{a}\mapsto\zeta_{p}\sqrt[p]{a}$. Next, $\eta{\omega}=\omega$ and,
$$
\frac{\bar{\bar{\sigma}}(\omega)}{\omega}=
\frac{\bar{\sigma}(\Phi(\beta))}{\Phi(\beta)}\cdot\Phi\left(\frac{\bar{\sigma}\sqrt[p]{a}}{\sqrt[p]{a}}\right)=
\frac{\Phi(\zeta_{p}\Nr_{L/K}(x))}{\Phi(x)^{p}}=
\left(\frac{\sqrt[p]{b}}{\Phi(x)}\right)^{p}\in{L(\sqrt[p]{b})^{*}}^{p},
$$
thus, $M/K$ is a Galois extension.
We extend $\bar{\bar{\sigma}}$ and $\eta$ to
$\tilde{\sigma},\tilde{\eta}\in\Gal(M/K)$ as in [6].
$\Gal(M/K)$ is generated by $\tilde{\sigma},\tilde{\eta}$ and $\tilde{\lambda}$,
where $\tilde{\lambda}(\sqrt[p]{\omega})=\zeta_{p}\sqrt[p]{\omega}$ - a generator for
$\Gal(M/L(\sqrt[p]{b}))$.
Clearly, $\tilde{\eta}$ is of order $p$.
Now,
\small
$$
    \tilde{\sigma}^{p}(\sqrt[p]{\omega})=
    \frac{b}{\Phi(x)\bar{\sigma}(\Phi(x))\cdots\bar{\sigma}^{p-1}(\Phi(x))}\sqrt[p]{\omega}=
    \frac{b}{\Phi(\Nr_{L/K}(x))}\sqrt[p]{\omega}=\Phi(\zeta_{p})\sqrt[p]{\omega}=\zeta_{p}^{-e^{p-2}}\sqrt[p]{\omega}.
$$
\normalsize
Hence, $\tilde{\sigma}$ is of order $p^{2}$ and $\tilde{\lambda}=\tilde{\sigma}^{p}$.
Finally, the relation $\tilde{\eta}\tilde{\sigma}=\tilde{\sigma}^{p+1}\tilde{\eta}$ holds.
It follows that the group isomorphism $\varphi:\Gal(M/K)\cong C_{p^{2}}\rtimes C_{p}$,
sending $\tilde{\sigma}\mapsto u$, $\tilde{\eta}\mapsto v$, is a solution isomorphism
for the given embedding problem.
\small
$$
    \frac{\bar{\tau}\omega}{\omega^e}=
    \frac{\bar{\tau}(\Phi(\beta))}{\Phi(\beta)^{e}}\cdot
    \frac{\bar{\tau}(\Phi(\sqrt[p]{a}))}{\Phi(\sqrt[p]{a})^{e}}=
    \left((\beta\sqrt[p]{a})^{\frac{1-e^{p-1}}{p}}\right)^{p}\in {L^{*}}^{p} \ \ \ \ ,
\    \frac{\bar{\tau}b}{b^{e}}=
    \left(\Nr_{L/K}(x)^{\frac{1-e^{p-1}}{p}}\right)^{p}\in {L^{*}}^{p}.
$$
\normalsize
Thus, $M/\QQ$ is a Galois extension.
We extend $\bar{\tau}$ to $\kappa\in\Gal(M/\QQ)$ by
$$
    \kappa(\sqrt[p]{\omega})=(\sqrt[p]{a}\beta)^{(1-e^{p-1})/p}\left({\sqrt[p]{\omega}}\right)^{e}, \ \ \ \ \
    \kappa(\sqrt[p]{b})=\Nr_{L/K}(x)^{(1-e^{p-1})/p}(\sqrt[p]{b})^{e}.
$$
The rest remains exactly as in the previous case.
\end{proof}
We end this section with the following observations.
\begin{theorem}
(a) \ Suppose that $L=\QQ(\zeta_{p^{2}})$. Then $x\in L^{*}$ induces an $H_{p^{3}}$-extension if and only if it induces a $C_{p^{2}}\rtimes C_{p}$-extension (in the ways described in Theorem 4.1 and Theorem 4.2).
\\
(b) \ Suppose that $L\neq\QQ(\zeta_{p^{2}})$. If an element $x\in L^{*}$ does not induce an $H_{p^{3}}$-extension, it necessarily induces a $C_{p^{2}}\rtimes C_{p}$-extension.
\\
(c) \ There are infinitely many elements $x\in L^{*}$ which induce both an $H_{p^{3}}$-extension and a $C_{p^{2}}\rtimes C_{p}$-extension.
\end{theorem}
\begin{proof}
 (a) and (b) follow from $\Phi(\zeta_{p})=\zeta_{p}^{-e^{p-2}}$. (c) follows from Corollary \ref{inf}.
\end{proof}
\section{Polynomials for the non-abelian groups of order $27$}
Following Ledet, we shall construct $E$-extensions and their polynomials over $\QQ$, induced by the elements $x$ we found in Example 3.6 and Example 3.7.
We replace the primitive root $e=2$ with $e=-1$ -
the homomorphism $\Phi$ does not depend on $e$ when we consider it modulo $3$ (it may not be the case for higher primes).
Clearly, $\kappa\sqrt[3]{\omega}=1/\sqrt[3]{\omega}$, thus $\bar{\tau}\omega=1/\omega$, so $\omega+1/\omega\in F$.
Set $X=\sqrt[3]{\omega}+1/\sqrt[3]{\omega}$. Then $X^{3}-3X=\omega+1/\omega$.
Hence, if $p(X)\in\QQ[X]$ is the minimal polynomial for $\omega+1/\omega$ over $\QQ$, then the splitting field
of $p(X^3-3X)$ over $\QQ$ is $F(\sqrt[3]{\omega}+1/\sqrt[3]{\omega})$, i.e, $p(X^3-3X)$ is an $E$-polynomial over $\QQ$.
\begin{example}
\emph{
Let $x$ be as in Example 3.7. In order to get an $H_{27}$-polynomial over $\QQ$ induced by this $x$, consider
$\omega=\Phi(x^{2}\bar\sigma{x})$. \
Using mathematical software (Maple), we get the following $\HH_{27}$-polynomial over $\QQ$:
$$
p(X^3-3X)=(X^3-3X)^{3}-\frac{3^4}{7^2}(X^3-3X)^2-\frac{3\cdot 37}{7^3}(X^3-3X)+\frac{1489}{7^4}.
$$
}
\end{example}
\begin{example}
\emph{
Due to the presence of the radical $\sqrt[p]{a}$,
explicit extensions for the $C_{p^{2}}\rtimes C_{p}$ are more complicated to describe (unless $L=\QQ(\zeta_{p^2})$).
\\
\indent Let $x$ be as in Example 3.6.
In order to get a $C_{9}\rtimes C_{3}$-polynomial over $\QQ$ induced by this
$x$, consider $\omega=\Phi(x^{2}\bar\sigma{x}\cdot\sqrt[3]{a})$,
where $L=K(\sqrt[3]{a})$. Denote $\delta=\delta_{3}(7)$ and consider
the element
$$
\theta=3\delta^2+3\delta+3\zeta_{3}\delta+\zeta_{3}-4.
$$
The minimal polynomial for $\delta$ over $\QQ$ is $X^3+X^2-2X-1$,
and $\sigma\delta=\delta^2-2$. It follows that
$\bar{\sigma}\theta=\zeta_{3}\theta$. Therefore, we can take
$\sqrt[3]{a}=\theta$. Using mathematical software (Maple), we get
the following $C_{9}\rtimes C_{3}$-polynomial over $\QQ$: $$
p(X^3-3X)=(X^3-3X)^{3}-\frac{2\cdot 3^2\cdot
29}{13^2}(X^3-3X)^2-\frac{3\cdot 5\cdot
373}{13^3}(X^3-3X)+\frac{6791}{13^3\cdot 7}. $$
}
\end{example}
\section{Acknowledgment}
The author is grateful to Jack Sonn, Moshe Roitman and Arne Ledet for useful discussions.
%
%%%%%%%%%%%%%%%%%%%%%%%%%%%
%      Bibliography
%%%%%%%%%%%%%%%%%%%%%%%%%%%
%

%
%      -------- Author --------
%
\vskip 1cm
\small
\noindent
Oz Ben-Shimol \\
Department of Mathematics \\
University of Haifa \\
Mount Carmel 31905, Haifa, Israel \\
E-mail Address: obenshim@math.haifa.ac.il
\end{document}